\documentclass[a4paper,11pt]{article}
\usepackage[multidot]{grffile}
\usepackage{latexsym,amssymb,enumerate,amsmath,epsfig,amsthm}
\usepackage[margin=1in]{geometry}
\usepackage{setspace,color}
\usepackage{graphics}
\usepackage{graphicx}
\usepackage[ruled]{algorithm2e}
\usepackage{empheq}
\usepackage{bm,amsmath}
\usepackage{subcaption}
\usepackage{hyperref}
\hypersetup{
    colorlinks=true,
    linkcolor=blue,
    filecolor=magenta,
    urlcolor=cyan,
    citecolor=red,
    pdfpagemode=FullScreen,
}

\usepackage{rotating,booktabs,array,amsmath,amssymb}
\newcolumntype{C}{>{$\displaystyle}c<{$}}
\usepackage{multirow}
\usepackage{lineno}

\newcommand{\bS}{\mathbb{S}}
\newcommand{\bR}{\mathbb{R}}
\newcommand{\bH}{\mathbb{H}}

\newcommand{\SLERP}{\operatorname{SLERP}}

\newcommand{\BCH}{\operatorname{BCH}}

\newcommand{\ImH}{\operatorname{Im}\bH}

\numberwithin{equation}{section}

\theoremstyle{plain}
\newtheorem{thm}{Theorem}[section]
\newtheorem{prop}[thm]{Proposition}
\newtheorem{lemma}[thm]{Lemma}
\newtheorem{corollary}[thm]{Corollary}

\theoremstyle{definition}
\newtheorem{definition}[thm]{Definition}
\newtheorem{assump}[thm]{Assumption}

\theoremstyle{remark}
\newtheorem{remark}[thm]{Remark}

\title{A Quaternion--BCH Framework for the Local Accuracy of SIDER Interpolation}
\author{
Shingyu Leung\thanks{Department of Mathematics, the Hong Kong University of Science and Technology, Clear Water Bay, Hong Kong. Email: {\bf masyleung@ust.hk}}
}
\date{}

\begin{document}

\maketitle

\begin{abstract}
Spherical Interpolation of orDER \(n\) (SIDER-\(n\)) is a recursive
high-order interpolation method for data on the unit sphere \(\bS^2\), built
from repeated spherical linear interpolation (SLERP).  This paper develops a
quaternion--Lie algebra framework for proving the local consistency of SIDER
for smooth spherical curves sampled at equally spaced parameter values.  Points
on \(\bS^2\) are represented as pure unit quaternions, and interpolation errors
are measured in fixed-base quaternion logarithmic coordinates.  In this
setting, each SLERP operation admits an exact Baker--Campbell--Hausdorff
(BCH) representation, which converts the geometric interpolation problem into
an algebraic problem involving filtered Lie-polynomial expansions.  The BCH
expansion shows that SLERP is affine to leading order, has no quadratic
correction, and has a first nonlinear correction that is cubic and
commutator-valued.  Using this structure, we prove that SIDER2 has a
third-order divided-error form with the same leading nodal factor as ordinary
quadratic interpolation.  We then show that the recursive SIDER step raises
the order by one: the affine part gives the Neville-type finite-difference
cancellation, while the nonlinear BCH remainder preserves the sharp filtered
degree structure after the nodal factor is removed.  Consequently, for every
fixed \(n\geq2\),
$
        d_{\bS^2}\bigl(\gamma(\theta h),P_i^{[n]}(\theta;h)\bigr)
        =
        O(h^{n+1})
$
under the stated smoothness and small-stencil assumptions.  The proof also
identifies the shift-invariance of the leading divided-error coefficient as
the algebraic compatibility condition underlying the SIDER recurrence.
\end{abstract}

\section{Introduction}

Interpolation of data constrained to a nonlinear manifold is a common task in
scientific computing.  In many applications the quantities of interest are not
arbitrary vectors in Euclidean space, but directions, orientations, rotations,
or normalized state variables.  A representative example is interpolation on
the unit sphere \(\bS^2\).  Applying standard polynomial interpolation to the
ambient Euclidean coordinates generally produces a curve that leaves the
sphere.  Projecting the curve back to \(\bS^2\) restores the constraint, but
such a projection changes the geometry of the interpolant and may affect its
formal approximation properties.  This motivates interpolation methods that
respect the spherical constraint throughout the construction.

Spherical interpolation is particularly important in computer graphics,
attitude dynamics, rigid-body motion, molecular simulation, and numerical
methods involving direction- or quaternion-valued unknowns.  Quaternions have
long been used as an efficient representation of three-dimensional rotations
and as a natural language for orientation interpolation
\cite{ham63,kui02,muk02,zha97}.  The classical spherical linear interpolation
method, or SLERP, joins two unit quaternions by a constant-speed geodesic
\cite{sho85}.  SLERP is geometrically natural and exactly preserves the unit
constraint, but a piecewise SLERP curve is only a low-order reconstruction and
has limited smoothness at the interpolation nodes.

In \cite{fonleu23}, we introduced a recursive family of spherical
interpolants called SIDER, standing for Spherical Interpolation of orDER
\(n\).  The construction is inspired by B\'ezier and Neville-type interpolation:
SIDER2 interpolates three spherical data points, SIDER3 combines two adjacent
SIDER2 interpolants, and higher SIDER levels are obtained recursively by
combining adjacent lower-level reconstructions.  The essential geometric
feature is that each affine interpolation step in the Euclidean construction
is replaced by SLERP.  SIDER also serves as the high-order building block in
the spherical essentially non-oscillatory method SENO, which selects among
candidate stencils to reduce oscillations near nonsmooth data
\cite{fonleu23}.

The numerical evidence in \cite{fonleu23} indicates that SIDER-\(n\) has the
expected local accuracy \(O(h^{n+1})\) for smooth data.  A direct proof in
normal coordinates is possible, but the calculations become increasingly
complicated because each SLERP operation contributes curvature-dependent
correction terms.  The purpose of the present paper is to give a different and
more algebraic proof using the quaternion representation.  The key observation
is that, in fixed-base quaternion logarithmic coordinates, SLERP has an exact
BCH representation:
\[
        \log\!\left(q_\theta^{-1}
        \SLERP(q_\theta\exp X,q_\theta\exp Y,\lambda)\right)
        =
        \BCH\!\left(X,\lambda\,\BCH(-X,Y)\right).
\]
Thus the nonlinear geometry of SLERP is encoded in the BCH series.  This
formulation makes transparent two important structural facts: there is no
quadratic correction in the local SLERP expansion, and the first nonlinear
correction is cubic and commutator-valued.

The proof proceeds by combining this BCH representation with a filtered
divided-error argument.  The base SIDER2 construction has the same leading
third-order error as quadratic interpolation, because the cubic BCH correction
vanishes when evaluated on the collinear first-order data.  The recursive
SIDER step then has the same algebraic form as a Neville finite-difference
cancellation.  Once a lower-level SIDER error is written in divided-error form,
the difference of adjacent divided-error coefficients produces an extra factor
of \(h\), and hence raises the order by one.  The nonlinear BCH remainder is
controlled in diagonal variables; its powers of the interpolation parameter are
paired with finite-difference degree savings, so the sharp filtered structure
is preserved through the recursion.

The analysis is local and asymptotic.  We assume that the data points and all
intermediate SIDER values remain in a common normal neighborhood of the exact
solution value.  This ensures that the quaternion logarithm is single-valued,
that SLERP uses a consistent local branch, and that the BCH expansions are
valid uniformly on bounded parameter intervals.  The paper focuses on equally
spaced parameter values; nonuniform stencils would require the corresponding
nonuniform Neville weights.

The remainder of the paper is organized as follows.  Section~2 reviews the
quaternion representation of spherical data and writes the SIDER construction
in quaternion notation.  Section~3 derives the exact BCH form of SLERP and its
cubic expansion.  Section~4 develops the filtered divided-error calculus used
in the proof.  Section~5 proves the SIDER2 base estimate.  Section~6 proves
the recursive order-raising lemma.  Section~7 combines these ingredients to
obtain the all-order local consistency theorem.  The final section summarizes
the result and discusses the scope of the argument.

\section{Quaternion representation and the SIDER construction}

We begin by introducing the quaternion notation used throughout the paper and
by rewriting the SIDER construction in this framework.  The purpose of this
section is twofold.  First, it fixes the algebraic setting in which SLERP can
be expressed through exponential and logarithm maps on the unit-quaternion
group.  Second, it records the SIDER2 and recursive SIDER formulas in a form
that is suitable for the Baker--Campbell--Hausdorff analysis developed in the
next sections.  The quaternion point of view is particularly useful because it
turns the geometric operation of spherical interpolation into an explicit
composition of Lie-group exponentials and logarithms.

Let \(\bH\) denote the algebra of quaternions.  We write a quaternion as
\[
        q=(a,\mathbf u)
        =
        a+u_1{\bf i}+u_2{\bf j}+u_3{\bf k},
        \qquad a\in\bR,\quad \mathbf u\in\bR^3 .
\]
The scalar part of \(q\) is \(a\), and the vector part is \(\mathbf u\).  The
Hamilton product of two quaternions is
\[
(a,\mathbf u)(b,\mathbf v)
=
(ab-\mathbf u\cdot\mathbf v,\,
a\mathbf v+b\mathbf u+\mathbf u\times\mathbf v).
\]
The Euclidean norm on \(\bH\cong\bR^4\) is multiplicative, and the unit
quaternions form a Lie group,
\[
        \bH_1=\{q\in\bH: |q|=1\}.
\]
The Lie algebra of \(\bH_1\) is the space of pure imaginary quaternions,
\[
        \ImH=\{(0,\mathbf u):\mathbf u\in\bR^3\}.
\]
We will identify \(\ImH\) with \(\bR^3\) whenever convenient.  Under this
identification, the Lie bracket is
\[
        [X,Y]=XY-YX=2X\times Y,
        \qquad X,Y\in\ImH .
\]
Thus noncommutativity of quaternion multiplication is represented by the
ordinary cross product in \(\bR^3\).  This observation is useful later, because
the BCH correction terms are nested commutators, and hence can be interpreted
geometrically as curvature-type cross-product corrections.

The exponential and logarithm maps on \(\bH_1\) have explicit forms.  If
\(X=(0,\mathbf x)\in\ImH\), then
\[
        \exp X
        =
        \left(\cos |\mathbf x|,
        \frac{\sin |\mathbf x|}{|\mathbf x|}\mathbf x\right),
\]
with the standard limiting interpretation at \(|\mathbf x|=0\).  Conversely,
if \(q=(a,\mathbf u)\in\bH_1\) is sufficiently close to the identity and
\(\mathbf u\neq0\), then
\[
        \log q
        =
        \left(0,
        \frac{\arccos a}{|\mathbf u|}\mathbf u\right).
\]
All logarithms in this paper are understood locally, using a fixed branch in a
small neighborhood of the identity.

A point \(p\in\bS^2\) is embedded into the unit quaternions as the pure unit
quaternion
\[
        q=(0,p)\in\bH_1 .
\]
For a smooth spherical curve \(\gamma(t)\in\bS^2\), we write
\[
        q(t)=(0,\gamma(t)).
\]
The set of pure unit quaternions is not a subgroup of \(\bH_1\).  Nevertheless,
SLERP between two pure unit quaternions remains in the two-dimensional real
linear span of the endpoints and therefore remains pure.  Consequently, the
quaternion formulation is compatible with the original interpolation problem
on \(\bS^2\).  We will use quaternion multiplication to represent the
geodesic-interpolation operation, but the resulting interpolants correspond to
points on the original sphere.

For two nearby unit quaternions \(q_a,q_b\), spherical linear interpolation is
written as
\begin{equation}
\label{eq:q_slerp}
        \SLERP(q_a,q_b,\lambda)
        =
        q_a\bigl(q_a^{-1}q_b\bigr)^\lambda
        =
        q_a\exp\!\left(\lambda\log(q_a^{-1}q_b)\right).
\end{equation}
When \(0\leq\lambda\leq1\), this is the constant-speed geodesic from \(q_a\)
to \(q_b\) in \(\bH_1\).  For bounded values of \(\lambda\) outside this
interval, the same formula gives a local geodesic extrapolation, provided the
logarithm branch is chosen consistently.  This extrapolatory use of SLERP is
essential in SIDER2, where the auxiliary control points are obtained by
extending geodesic segments beyond the middle data point.

The advantage of \eqref{eq:q_slerp} is that it is a Lie-group formula.  If
\(q_a\) and \(q_b\) are both close to a reference value \(q_\theta\), we can
write
\[
        q_a=q_\theta\exp X,
        \qquad
        q_b=q_\theta\exp Y,
        \qquad X,Y\in\ImH .
\]
Then the left-trivialized logarithm of the SLERP point is
\[
\log\!\left(
q_\theta^{-1}\SLERP(q_a,q_b,\lambda)
\right)
=
\log\!\left[
\exp X
\exp\!\left(
\lambda\log(\exp(-X)\exp Y)
\right)
\right].
\]
This identity is the starting point of the BCH analysis.  It expresses SLERP
entirely in terms of exponential, logarithm, and multiplication in the
unit-quaternion group.  The local expansion of SLERP therefore follows from
the Baker--Campbell--Hausdorff formula rather than from direct trigonometric
expansions on the sphere.

We impose the following local assumption throughout the analysis.

\begin{assump}[Small-stencil condition]
\label{assump:small_stencil_q}
For the fixed interpolation level under consideration, \(h>0\) is sufficiently
small that all data points, extrapolated control points, and intermediate
SIDER values remain in a common normal neighborhood of the exact value
\(q(\theta h)\).  Consequently, all quaternion logarithms below are
single-valued, all SLERP operations use the same local branch, and all
constants in the estimates are uniform for \(\theta\) in bounded intervals.
\end{assump}

This assumption is local and asymptotic.  It excludes antipodal ambiguity and
large angular jumps, which are intrinsic difficulties in spherical and
quaternion interpolation.  It also ensures that logarithmic coordinates provide
a valid way to compare the exact curve and the reconstructed values.

Let
\[
        p_j=\gamma(jh),\qquad q_j=q(jh).
\]
For an evaluation point
\[
        s=\theta h,
\]
set
\[
        q_\theta=q(\theta h).
\]
The error of a SIDER candidate \(Q(\theta;h)\) is measured in
left-trivialized logarithmic coordinates:
\[
        E(\theta;h)
        =
        \log\bigl(q_\theta^{-1}Q(\theta;h)\bigr)
        \in \ImH .
\]
Since left multiplication by \(q_\theta^{-1}\) is an isometry on \(\bH_1\),
and since the logarithm map is locally distance-preserving to leading order,
bounds on \(E(\theta;h)\) imply the corresponding local interpolation error
bounds for the spherical curve.  More precisely, in the common normal
neighborhood,
\[
        d_{\bH_1}\bigl(q_\theta,Q(\theta;h)\bigr)
        =
        \bigl|\log(q_\theta^{-1}Q(\theta;h))\bigr|,
\]
and this distance is equivalent to the spherical distance between the
corresponding points on \(\bS^2\).

We now recall the SIDER construction in this notation.  For three data points
\[
        q_i,\quad q_{i+1},\quad q_{i+2},
\]
SIDER2 first defines two extrapolated control points by
\[
        d_a=\SLERP(q_{i+2},q_{i+1},2),
        \qquad
        d_b=\SLERP(q_i,q_{i+1},2).
\]
The point \(d_a\) is obtained by moving from \(q_{i+2}\) toward \(q_{i+1}\)
and continuing the same geodesic beyond \(q_{i+1}\).  Similarly, \(d_b\) is
obtained by moving from \(q_i\) toward \(q_{i+1}\) and extrapolating beyond
\(q_{i+1}\).  These auxiliary points play the role of spherical control points.
They are chosen so that the final construction passes through all three data
points, rather than treating the middle point merely as a B\'ezier control point.

Let
\[
        \xi=\theta-i.
\]
The SIDER2 value on the stencil
\(\{q_i,q_{i+1},q_{i+2}\}\) is
\begin{equation}
\label{eq:q_sider2}
        P_i^{[2]}(\theta;h)
        =
        \SLERP\!\left(
        \SLERP(q_i,d_a,\xi/2),
        \SLERP(d_b,q_{i+2},\xi/2),
        \xi/2
        \right).
\end{equation}
This construction interpolates the three nodes
\[
        \xi=0,\qquad \xi=1,\qquad \xi=2.
\]
Indeed, at \(\xi=0\) and \(\xi=2\) the formula returns the endpoint data
\(q_i\) and \(q_{i+2}\), while the symmetric choice of the extrapolated control
points forces the value at \(\xi=1\) to be \(q_{i+1}\).  If all SLERP
operations in \eqref{eq:q_sider2} are replaced by affine interpolation in a
vector space, the formula reduces to the usual quadratic interpolation
construction.

For \(k\geq3\), the higher-order SIDER reconstruction is defined recursively.
Let \(P_i^{[k]}(\theta;h)\) denote the SIDER-\(k\) interpolant associated with
the stencil
\[
        q_i,q_{i+1},\ldots,q_{i+k}.
\]
Then
\begin{equation}
\label{eq:q_sider_rec}
        P_i^{[k]}(\theta;h)
        =
        \SLERP\!\left(
        P_i^{[k-1]}(\theta;h),
        P_{i+1}^{[k-1]}(\theta;h),
        \frac{\xi}{k}
        \right),
        \qquad \xi=\theta-i .
\end{equation}
This is the spherical counterpart of the equally spaced Neville recursion.
The parameter \(\xi/k\) is the same normalized weight that appears in the
Euclidean recursion, while the affine blending operation is replaced by
SLERP.  Consequently, the SIDER construction preserves the spherical
constraint at every stage, but its error analysis must account for the
noncommutativity and curvature effects encoded in the quaternion logarithms.

The logarithmic error associated with \(P_i^{[k]}\) is
\[
        E_i^{[k]}(\theta;h)
        =
        \log\bigl(q_\theta^{-1}P_i^{[k]}(\theta;h)\bigr).
\]
The rest of the paper studies \(E_i^{[k]}(\theta;h)\) as a formal and
asymptotic power series in \(h\).  The interpolation property gives zeros of
this error at the stencil nodes.  The quaternion--BCH representation of SLERP
then allows us to track how the polynomial dependence on the shifted variable
\(\xi=\theta-i\) is propagated through the SIDER recursion.  This combination
of nodal interpolation and Lie-algebraic structure is the basis for the
all-order consistency result proved below.

\section{BCH form of SLERP}

The central advantage of the quaternion formulation is that SLERP can be
written exactly in fixed-base logarithmic coordinates.  This converts the
geometric interpolation operation into an algebraic operation in the Lie
algebra \(\ImH\).  The resulting expression is particularly well suited for
local error analysis, because it allows all nonlinear effects of SLERP to be
organized by the Baker--Campbell--Hausdorff (BCH) formula.  In this section we
derive the exact BCH representation of SLERP, compute its cubic expansion, and
record two consequences that will be used throughout the analysis of the SIDER
recursion.

Let \(q_\theta\) be the exact quaternion value at the evaluation point
\(s=\theta h\).  Suppose that two nearby unit quaternions \(Q_A,Q_B\) are
written in left-trivialized logarithmic coordinates as
\[
        Q_A=q_\theta\exp X,
        \qquad
        Q_B=q_\theta\exp Y,
        \qquad X,Y\in\ImH .
\]
The vectors \(X\) and \(Y\) represent the local errors of the two endpoints
relative to the same base value \(q_\theta\).  The goal is to express the
logarithmic coordinate of
\[
        \SLERP(Q_A,Q_B,\lambda)
\]
relative to the same base point.  This gives a nonlinear binary operation on
the Lie algebra, depending on the interpolation parameter \(\lambda\).

\begin{lemma}[BCH representation of SLERP]
\label{lem:bch_slerp}
Let \(Q_A=q_\theta\exp X\) and \(Q_B=q_\theta\exp Y\), where
\(X,Y\in\ImH\) are sufficiently small.  Then
\[
\log\!\left(q_\theta^{-1}\SLERP(Q_A,Q_B,\lambda)\right)
=
\mathcal S_\lambda(X,Y),
\]
where
\begin{equation}
\label{eq:S_lambda_BCH}
        \mathcal S_\lambda(X,Y)
        =
        \BCH\!\left(X,\lambda\,\BCH(-X,Y)\right).
\end{equation}
Here
\[
        \BCH(A,B)=\log(\exp A\exp B)
\]
denotes the Baker--Campbell--Hausdorff series.
\end{lemma}

\begin{proof}
By the quaternion formula for SLERP,
\[
\SLERP(Q_A,Q_B,\lambda)
=
Q_A\exp\!\left(\lambda\log(Q_A^{-1}Q_B)\right).
\]
Since
\[
        Q_A=q_\theta\exp X,
        \qquad
        Q_B=q_\theta\exp Y,
\]
we have
\[
        Q_A^{-1}Q_B
        =
        \exp(-X)q_\theta^{-1}q_\theta\exp Y
        =
        \exp(-X)\exp Y .
\]
Therefore
\[
q_\theta^{-1}\SLERP(Q_A,Q_B,\lambda)
=
\exp X\,
\exp\!\left(\lambda\log(\exp(-X)\exp Y)\right).
\]
Taking the logarithm of both sides gives
\[
\log\!\left(q_\theta^{-1}\SLERP(Q_A,Q_B,\lambda)\right)
=
\BCH\!\left(
X,
\lambda\,\BCH(-X,Y)
\right),
\]
which is \eqref{eq:S_lambda_BCH}.
\end{proof}

Lemma~\ref{lem:bch_slerp} is exact, as long as all logarithms are taken on the
same local branch.  It shows that the local behavior of SLERP is governed by a
universal Lie-polynomial map.  In particular, the leading term of
\(\mathcal S_\lambda(X,Y)\) should be the affine combination
\[
        (1-\lambda)X+\lambda Y,
\]
while the non-affine corrections are determined by commutators.  The next
lemma computes the first nontrivial correction explicitly.

\begin{lemma}[Cubic BCH expansion]
\label{lem:cubic_bch}
For bounded \(\lambda\) and sufficiently small \(X,Y\in\ImH\),
\begin{equation}
\label{eq:cubic_bch}
\mathcal S_\lambda(X,Y)
=
(1-\lambda)X+\lambda Y
+
\frac{\lambda(1-\lambda)}{12}
        [Y-X,[X,Y]]
+
O\bigl((|X|+|Y|)^4\bigr).
\end{equation}
In particular, the quadratic term vanishes.
\end{lemma}

\begin{proof}
We use the BCH expansion through homogeneous degree three:
\[
\BCH(A,B)
=
A+B+\frac12[A,B]
+\frac1{12}[A,[A,B]]
+\frac1{12}[B,[B,A]]
+O(4),
\]
where \(O(4)\) denotes terms of homogeneous degree at least four in \(A\) and
\(B\).  Applying this first to \(\BCH(-X,Y)\), we obtain
\[
\begin{aligned}
\BCH(-X,Y)
&=
Y-X+\frac12[-X,Y]
+\frac1{12}[-X,[-X,Y]]
+\frac1{12}[Y,[Y,-X]]
+O(4)  \\
&=
Y-X-\frac12[X,Y]
+\frac1{12}[X,[X,Y]]
+\frac1{12}[Y,[X,Y]]
+O(4).
\end{aligned}
\]
Set
\[
        Z=\BCH(-X,Y).
\]
Then
\[
        \mathcal S_\lambda(X,Y)=\BCH(X,\lambda Z).
\]
Substituting the expansion of \(Z\) into the BCH formula again gives
\[
\BCH(X,\lambda Z)
=
X+\lambda Z
+\frac12[X,\lambda Z]
+\frac1{12}[X,[X,\lambda Z]]
+\frac1{12}[\lambda Z,[\lambda Z,X]]
+O(4).
\]
The homogeneous degree-one part is
\[
        X+\lambda(Y-X)=(1-\lambda)X+\lambda Y.
\]
The homogeneous degree-two terms are
\[
        -\frac{\lambda}{2}[X,Y]
        +
        \frac{\lambda}{2}[X,Y-X],
\]
which cancel because \([X,Y-X]=[X,Y]\).  Thus there is no quadratic
correction.

It remains to collect terms of degree three.  These come from the cubic part
of \(\lambda Z\), the bracket of \(X\) with the quadratic part of
\(\lambda Z\), and the two cubic BCH terms involving the degree-one part of
\(Z\).  Combining these contributions gives
\[
\begin{aligned}
&\frac{\lambda}{12}[X,[X,Y]]
+\frac{\lambda}{12}[Y,[X,Y]]
-\frac{\lambda}{4}[X,[X,Y]]
+\frac{\lambda}{12}[X,[X,Y]]       \\
&\qquad
-\frac{\lambda^2}{12}[Y-X,[X,Y]] .
\end{aligned}
\]
After simplification, the coefficient of \([X,[X,Y]]\) is
\[
        -\frac{\lambda(1-\lambda)}{12},
\]
and the coefficient of \([Y,[X,Y]]\) is
\[
        \frac{\lambda(1-\lambda)}{12}.
\]
Hence the cubic term is
\[
        \frac{\lambda(1-\lambda)}{12}
        \left(
        [Y,[X,Y]]-[X,[X,Y]]
        \right)
        =
        \frac{\lambda(1-\lambda)}{12}
        [Y-X,[X,Y]].
\]
This proves \eqref{eq:cubic_bch}.
\end{proof}

The cubic expansion has two important interpretations.  First, SLERP is affine
to second order in fixed-base logarithmic coordinates: the quadratic terms
cancel exactly.  Second, the first nonlinear correction is a commutator term.
For pure imaginary quaternions, commutators are cross products.  Hence this
correction vanishes whenever \(X\) and \(Y\) are collinear.  This observation
will explain why the leading SIDER2 error agrees with the Euclidean quadratic
interpolation error: at the lowest relevant order, the local data lie along the
tangent direction of the curve, so the commutator correction does not
contribute.

We next record the stability consequence that will be used repeatedly in the
recursive SIDER analysis.  Once two candidate reconstructions are already
high-order accurate, the final SLERP operation between them behaves affinely to
all orders relevant for the next cancellation.

\begin{lemma}[Affine stability]
\label{lem:q_affine_stability}
Let \(X(h),Y(h)\in\ImH\) satisfy
\[
        X(h)=O(h^m),
        \qquad
        Y(h)=O(h^m),
\]
for some \(m\geq1\).  Then, uniformly for bounded \(\lambda\),
\[
        \mathcal S_\lambda(X(h),Y(h))
        =
        (1-\lambda)X(h)+\lambda Y(h)+O(h^{3m}).
\]
\end{lemma}

\begin{proof}
By Lemma~\ref{lem:cubic_bch}, the difference between
\(\mathcal S_\lambda(X,Y)\) and the affine combination
\((1-\lambda)X+\lambda Y\) begins with a homogeneous cubic expression in
\(X\) and \(Y\).  If \(X(h),Y(h)=O(h^m)\), this cubic expression is
\(O(h^{3m})\).  The higher-order BCH terms are of still higher order.  The
estimate is uniform for bounded \(\lambda\), since the coefficients in the BCH
expansion are smooth functions of \(\lambda\) on bounded intervals.
\end{proof}

For later use, we also formulate the higher-order structure of
\(\mathcal S_\lambda\).  The exact coefficients of the higher-order terms will
not be needed, but their algebraic form is essential.

\begin{lemma}[Formal Lie-polynomial structure]
\label{lem:q_formal_structure}
For every integer \(M\geq1\), the map
\[
        \mathcal S_\lambda(X,Y)
        =
        \BCH\!\left(X,\lambda\,\BCH(-X,Y)\right)
\]
has an expansion
\[
        \mathcal S_\lambda(X,Y)
        =
        (1-\lambda)X+\lambda Y
        +
        \sum_{r=3}^{M}\mathcal L_r(X,Y;\lambda)
        +
        O\bigl((|X|+|Y|)^{M+1}\bigr),
\]
where each \(\mathcal L_r\) is a homogeneous Lie polynomial of degree \(r\) in
\(X\) and \(Y\).  The coefficients of \(\mathcal L_r\) are polynomials in
\(\lambda\) of degree at most \(r\).  There is no homogeneous term of degree
two.
\end{lemma}

\begin{proof}
The BCH series is a universal Lie series: at each homogeneous degree it is a
finite linear combination of nested commutators of its arguments with rational
coefficients.  Applying this first to \(\BCH(-X,Y)\) gives a Lie series in
\(X\) and \(Y\).  Multiplication by \(\lambda\) multiplies each homogeneous
component by one power of \(\lambda\).  Applying BCH again to
\[
        X
        \quad\text{and}\quad
        \lambda\,\BCH(-X,Y)
\]
produces another Lie series in \(X\) and \(Y\).  At homogeneous degree \(r\),
the coefficient is obtained from products of at most \(r\) factors involving
\(\lambda\), and hence is a polynomial in \(\lambda\) of degree at most \(r\).

The absence of a quadratic term was established explicitly in
Lemma~\ref{lem:cubic_bch}.  The remainder estimate follows from the local
convergence, or equivalently the asymptotic validity, of the BCH expansion in
a sufficiently small neighborhood of the identity, uniformly for bounded
\(\lambda\).
\end{proof}

A final observation will be useful for interpreting the SIDER2 calculation.
Every nonlinear term in the BCH expansion is built from commutators.  If
\(X\) and \(Y\) are scalar multiples of the same Lie-algebra element, then
\[
        [X,Y]=0,
\]
and all commutator terms vanish.  In that case
\[
        \mathcal S_\lambda(X,Y)=(1-\lambda)X+\lambda Y
\]
exactly along the corresponding one-parameter subgroup.  Thus the
non-affine part of SLERP measures the failure of the two endpoint errors to
lie on the same local geodesic.  This is the quaternion-algebraic analogue of
the curvature correction in normal coordinates, and it is the mechanism that
must be controlled in the high-order SIDER proof.

\section{Filtered divided-error calculus}
\label{sec:filtered_calculus}

The proof of the all-order consistency result is based on a filtered analogue
of the divided-error structure in classical interpolation.  In Euclidean
Neville interpolation, the error is controlled by two elementary facts: the
interpolant agrees with the data at the interpolation nodes, and the
coefficient of each term in the Taylor expansion has a polynomial degree that
is compatible with the order of that term.  The same mechanism is used here,
but the errors are measured in quaternion logarithmic coordinates and the
interpolation operation is SLERP rather than affine interpolation.

This section introduces the bookkeeping needed for the recursive proof.  The
variable
\[
        \xi=\theta-i
\]
measures the location of the evaluation point relative to the left endpoint of
the stencil
\[
        q_i,q_{i+1},\ldots,q_{i+k}.
\]
Thus the local interpolation nodes for a SIDER-\(k\) reconstruction are
\[
        \xi=0,1,\ldots,k.
\]
For \(k\geq0\), define the corresponding nodal product
\[
        \Pi_k(\xi)=\prod_{m=0}^{k}(\xi-m).
\]
The factor \(\Pi_k\) will appear repeatedly, because the interpolation error
of a SIDER-\(k\) reconstruction vanishes at these \(k+1\) nodes.

\begin{definition}[Filtered polynomial series]
\label{def:filtered_series}
A formal series
\[
        C(\theta,\xi,h)
        \sim
        \sum_{r=0}^{\infty}h^r C_r(\theta,\xi)
\]
is called degree-filtered if, for each \(r\), the coefficient
\(C_r(\theta,\xi)\) is a polynomial in \(\xi\) of degree at most \(r\).  The
coefficients of these polynomials may depend smoothly on the local quaternion
jet of the curve \(q\) at the reference value \(q(\theta h)\).
\end{definition}

The dependence on \(\theta\) is treated as smooth but fixed in the degree
count.  In other words, the filtration only measures polynomial dependence on
the shifted stencil variable \(\xi\).  The local jet of \(q\) may enter through
coefficients such as logarithmic velocity, acceleration, and higher
left-trivialized derivatives, but these quantities do not affect the degree in
\(\xi\).

A basic operation in the recursive proof is a shifted finite difference.  The
following elementary lemma records the stability of the filtered class under
such differences.

\begin{lemma}[Filtered finite differences]
\label{lem:filtered_difference}
Let
\[
        C(\theta,\xi,h)
        \sim
        \sum_{r=0}^{\infty}h^r C_r(\theta,\xi)
\]
be degree-filtered.  Suppose that
\[
        C(\theta,\xi,h)-C(\theta,\xi-1,h)=O(h)
\]
as a formal series.  Then
\[
        \frac{C(\theta,\xi,h)-C(\theta,\xi-1,h)}{h}
\]
is again degree-filtered.
\end{lemma}

\begin{proof}
Since \(C\) is degree-filtered, the leading coefficient
\(C_0(\theta,\xi)\) has degree at most zero in \(\xi\), and therefore is
independent of \(\xi\).  Hence the \(h^0\) term in the numerator vanishes.
After division by \(h\), the coefficient of \(h^r\) is
\[
        C_{r+1}(\theta,\xi)-C_{r+1}(\theta,\xi-1).
\]
The polynomial \(C_{r+1}\) has degree at most \(r+1\).  A finite difference in
\(\xi\) lowers the degree by at least one, so the resulting coefficient has
degree at most \(r\).  This is precisely the degree-filtered condition.
\end{proof}

We will say that the SIDER-\(k\) logarithmic error has a filtered
divided-error form if
\begin{equation}
\label{eq:filtered_divided_error_form}
        E_i^{[k]}(\theta;h)
        =
        h^{k+1}\Pi_k(\xi)\,C_k(\theta,\xi,h),
        \qquad \xi=\theta-i,
\end{equation}
where \(C_k\) is degree-filtered.  This form contains two pieces of
information.  The factor \(h^{k+1}\) gives the expected local order of
accuracy.  The factor \(\Pi_k(\xi)\) encodes the interpolation conditions at
the \(k+1\) nodes.  The remaining filtered coefficient \(C_k\) contains the
higher-order local information.  In particular, its leading term is
independent of \(\xi\), which is the coefficient-compatibility property needed
for the Neville-type cancellation in the next SIDER level.

The following polynomial division lemma will be used to extract nodal factors
from interpolation conditions.

\begin{lemma}[Division by the nodal factor]
\label{lem:hadamard_nodal}
Let
\[
        F(\xi,h)\sim \sum_{r=r_0}^{\infty}h^r F_r(\xi)
\]
be a formal series whose coefficients \(F_r\) are polynomials in \(\xi\).
Assume that
\[
        F(\ell,h)=0,\qquad \ell=0,1,\ldots,k,
\]
as a formal series.  Then
\[
        F(\xi,h)=\Pi_k(\xi)D(\xi,h)
\]
for a formal series \(D\).  If, moreover,
\[
        \deg_\xi F_r\leq r,
\]
then the coefficient of \(h^r\) in \(D\) has degree at most \(r-k-1\).
\end{lemma}

\begin{proof}
The condition \(F(\ell,h)=0\) for each node means that, for every \(r\),
\[
        F_r(\ell)=0,\qquad \ell=0,1,\ldots,k.
\]
Hence each polynomial \(F_r\) is divisible by
\[
        \Pi_k(\xi)=\prod_{m=0}^{k}(\xi-m).
\]
Dividing coefficient by coefficient gives the factorization
\[
        F(\xi,h)=\Pi_k(\xi)D(\xi,h).
\]
If \(\deg_\xi F_r\leq r\), then division by the degree-\((k+1)\) polynomial
\(\Pi_k\) gives a quotient coefficient of degree at most \(r-k-1\).  When
\(r<k+1\), the corresponding quotient coefficient is zero.
\end{proof}

The preceding lemma is the algebraic reason that interpolation at sufficiently
many nodes enforces high-order consistency.  If the coefficient of \(h^r\) in
an error expansion has degree at most \(r\), and if the error vanishes at
\(k+1\) distinct nodes, then all coefficients with \(r\leq k\) must vanish.
The first possible nonzero term is therefore of order \(h^{k+1}\), and it
must contain the nodal factor \(\Pi_k\).  The remaining task is to show that
the SIDER construction preserves this degree structure.  The base case is
SIDER2, which is established next.

\section{The SIDER2 base estimate}
\label{sec:sider2_base_filtered}

We now prove the base estimate for the recursive argument.  The result shows
that SIDER2 has third-order local accuracy and, more importantly for the
higher-order proof, that its logarithmic error has the filtered divided-error
structure required by \eqref{eq:filtered_divided_error_form}.  This section
also identifies the leading coefficient of the SIDER2 error, which agrees with
the leading coefficient of ordinary quadratic interpolation in logarithmic
coordinates.

Fix
\[
        s=\theta h,
        \qquad
        q_\theta=q(s),
\]
and define the logarithmic data
\[
        X_j(\theta,h)
        =
        \log\bigl(q_\theta^{-1}q_j\bigr)
        =
        \log\bigl(q(\theta h)^{-1}q(jh)\bigr).
\]
If \(q\) is sufficiently smooth, then the left-trivialized logarithmic Taylor
expansion gives, for every fixed truncation order \(M\),
\begin{equation}
\label{eq:log_data_taylor}
        X_j(\theta,h)
        =
        \sum_{\ell=1}^{M}
        \frac{(j-\theta)^\ell h^\ell}{\ell!}\Omega_\ell(\theta)
        +O(h^{M+1}),
\end{equation}
where
\[
        \Omega_\ell(\theta)\in\ImH
\]
are the left logarithmic Taylor coefficients of the curve at \(q(\theta h)\).
For the shifted stencil \(j=i+m\), we have
\[
        j-\theta=m-\xi,
        \qquad
        \xi=\theta-i.
\]
Thus the coefficient of \(h^\ell\) in \(X_{i+m}\) is a polynomial in
\(\xi\) of degree \(\ell\).  This is the source of the degree filtration.

The SIDER2 construction involves five SLERP operations: two extrapolations to
construct the control points, two inner SLERPs, and one final SLERP.  In
left-trivialized logarithmic coordinates, each of these operations is expressed
by the BCH map \(\mathcal S_\lambda\) from Lemma~\ref{lem:bch_slerp}.  Since
\(\mathcal S_\lambda\) has a Lie-polynomial expansion whose coefficients are
polynomial functions of \(\lambda\), the SIDER2 logarithmic error admits a
formal expansion in powers of \(h\).  The next proposition gives the precise
form needed for the induction.

\begin{prop}[Filtered SIDER2 base estimate]
\label{prop:q_sider2_base}
Assume that \(q\) is sufficiently smooth and that the small-stencil condition
in Assumption~\ref{assump:small_stencil_q} holds.  Then the SIDER2 error has
the expansion
\begin{equation}
\label{eq:q_sider2_error}
        E_i^{[2]}(\theta;h)
        =
        h^3\Pi_2(\xi)\,C_2(\theta,\xi,h),
        \qquad \xi=\theta-i,
\end{equation}
where \(C_2\) is degree-filtered.  Moreover,
\[
        C_2(\theta,\xi,h)
        =
        -\frac16\Omega_3(\theta)+O(h).
\]
Consequently,
\[
        d_{\bS^2}\bigl(\gamma(\theta h),P_i^{[2]}(\theta;h)\bigr)
        =
        O(h^3),
\]
uniformly for \(\xi\) in bounded intervals.
\end{prop}

\begin{proof}
We work in logarithmic coordinates based at \(q_\theta\).  After left
translation by \(q_\theta^{-1}\), the exact value becomes the identity
quaternion and the data on the stencil are
\[
        \exp X_i,\qquad
        \exp X_{i+1},\qquad
        \exp X_{i+2}.
\]
The SIDER2 formula \eqref{eq:q_sider2} is a finite composition of the map
\(\mathcal S_\lambda\), with parameter \(\lambda=2\) for the two
control-point extrapolations and parameter \(\lambda=\xi/2\) for the three
B\'ezier-type SLERP operations.  Since BCH is analytic near the identity, this
composition has a formal expansion in powers of \(h\).

We first determine the leading term.  Keeping only the affine part of every
SLERP operation gives exactly the Euclidean quadratic interpolation formula
applied to the logarithmic data
\[
        X_i,\quad X_{i+1},\quad X_{i+2},
\]
evaluated at the local coordinate \(\xi\).  The exact value in these
coordinates is zero.  The affine contribution therefore reproduces the
linear and quadratic terms in \eqref{eq:log_data_taylor}.  Its first possible
nonzero contribution is the ordinary quadratic interpolation remainder:
\[
        -\frac{h^3}{6}
        \Pi_2(\xi)\Omega_3(\theta).
\]

It remains to check that the first nonlinear BCH correction does not add
another term of order \(h^3\).  By Lemma~\ref{lem:cubic_bch}, the first
non-affine term in a SLERP operation is cubic and is built from commutators of
the endpoint logarithms.  At order \(h^3\), this cubic expression only sees
the first-order parts of the logarithmic data:
\[
        X_{i+m}^{(1)}
        =
        (m-\xi)h\,\Omega_1(\theta),
        \qquad m=0,1,2.
\]
These vectors are all scalar multiples of the same Lie-algebra element
\(\Omega_1(\theta)\).  Hence their commutators vanish.  Therefore the cubic
BCH correction contributes no \(h^3\) term to the SIDER2 error.  We obtain
\[
        E_i^{[2]}(\theta;h)
        =
        -\frac{h^3}{6}
        \Pi_2(\xi)\Omega_3(\theta)
        +O(h^4).
\]

We next record the full filtered structure.  The logarithmic Taylor expansion
\eqref{eq:log_data_taylor} gives coefficients of degree at most their
corresponding power of \(h\).  The BCH representation expresses each SLERP
operation as a finite sum, at each homogeneous order, of nested commutators of
the endpoint logarithms.  The only possible source of extra powers of \(\xi\)
is the parameter \(\lambda=\xi/2\) in the inner and final SLERP operations.
For the SIDER2 control-point geometry, these powers are paired with diagonal
differences of the corresponding endpoint logarithms.  Those differences have
one degree saving in \(\xi\), exactly as in the diagonal-variable argument used
later in Lemma~\ref{lem:filtered_bch_closure}.  Hence the coefficient of
\(h^r\) in the SIDER2 error is a polynomial in \(\xi\) of degree at most
\(r\).  This is the degree filtration in
Definition~\ref{def:filtered_series}.

The SIDER2 reconstruction interpolates the three stencil nodes.  Hence
\[
        E_i^{[2]}(\theta;h)=0
        \qquad\text{for}\qquad
        \xi=0,1,2.
\]
Applying Lemma~\ref{lem:hadamard_nodal} with \(k=2\), the error is divisible
coefficient by coefficient by
\[
        \Pi_2(\xi)=\xi(\xi-1)(\xi-2).
\]
Thus
\[
        E_i^{[2]}(\theta;h)=\Pi_2(\xi)D_2(\theta,\xi,h).
\]
Since the first nonzero term begins at order \(h^3\), we may write
\[
        D_2(\theta,\xi,h)=h^3 C_2(\theta,\xi,h).
\]
The degree estimate in Lemma~\ref{lem:hadamard_nodal} shows that \(C_2\) is
degree-filtered.  The leading coefficient is obtained from the leading term
already computed:
\[
        C_2(\theta,\xi,h)
        =
        -\frac16\Omega_3(\theta)+O(h).
\]
This proves \eqref{eq:q_sider2_error}.

Finally, inside the common normal neighborhood, the spherical interpolation
error is locally equivalent to the norm of the quaternion logarithmic error.
Equivalently,
\[
        d_{\bS^2}\bigl(\gamma(\theta h),P_i^{[2]}(\theta;h)\bigr)
        \leq C\left|E_i^{[2]}(\theta;h)\right|
\]
for \(h\) sufficiently small.  Since \(\Pi_2(\xi)\) is bounded for \(\xi\) in
any fixed bounded interval, the \(O(h^3)\) distance estimate follows.
\end{proof}

\begin{remark}
The proposition explains why SIDER2 has the same leading error as ordinary
quadratic interpolation.  The affine part of the construction is precisely the
quadratic interpolant in logarithmic coordinates.  The first nonlinear
correction introduced by SLERP is commutator-valued, and at the lowest
relevant order the data are collinear multiples of the same tangent element
\(\Omega_1(\theta)\).  Thus this correction vanishes at order \(h^3\).
Curvature-dependent and noncommutative effects enter at higher orders, where
they are absorbed into the filtered coefficient \(C_2\).
\end{remark}

\section{The recursive order-raising step}
\label{sec:recursive_order_raising}

We now prove the mechanism by which the SIDER recursion increases the local
order by one.  This is the quaternion--BCH analogue of the classical
order-raising step in Neville interpolation.  The key point is that, once two
adjacent lower-level SIDER candidates have already been written in filtered
divided-error form, the final SLERP operation between them behaves affinely to
the order needed for the next cancellation.  The affine part then produces a
finite difference of the divided-error coefficient.  Since the leading
coefficient of a degree-filtered series is independent of the shifted stencil
variable, this finite difference contributes one additional factor of \(h\).

There are two ingredients in the proof.  The first is the explicit
Neville-type cancellation in the affine part of SLERP.  This part is purely
algebraic and produces the new nodal factor \(\Pi_k(\xi)\).  The second is the
control of the nonlinear BCH remainder.  Although this remainder is generated
by the noncommutativity of the quaternion algebra, it automatically preserves
the sharp filtered degree structure.  The reason is that the nonlinear part of
SLERP is at least quadratic in the diagonal difference between the two inputs.
This gives exactly the degree saving needed to compensate for the powers of
the interpolation parameter \(\lambda=\xi/k\).

We first prove this filtered closure property.

\begin{lemma}[Filtered closure of the nonlinear BCH remainder]
\label{lem:filtered_bch_closure}
Let \(k\geq3\).  Suppose
\[
        X(\xi,h)
        =
        h^k\Pi_{k-1}(\xi)C(\theta,\xi,h),
\]
and
\[
        Y(\xi,h)
        =
        h^k\Pi_{k-1}(\xi-1)C(\theta,\xi-1,h),
\]
where \(C\) is degree-filtered.  Set
\[
        \lambda=\frac{\xi}{k}
\]
and define the nonlinear SLERP remainder by
\[
        R_\lambda(X,Y)
        =
        \mathcal S_\lambda(X,Y)
        -
        \bigl((1-\lambda)X+\lambda Y\bigr).
\]
Then
\[
        R_{\xi/k}(X,Y)
        =
        h^{k+1}\Pi_k(\xi)\mathcal R_k(\theta,\xi,h),
\]
where \(\mathcal R_k\) is degree-filtered.
\end{lemma}

\begin{proof}
We use the diagonal variable
\[
        \Delta=Y-X.
\]
Thus \(Y=X+\Delta\), and
\[
        R_\lambda(X,Y)
        =
        \mathcal S_\lambda(X,X+\Delta)-X-\lambda\Delta .
\]
Define
\[
        \mathcal N_\lambda(X,\Delta)
        =
        \mathcal S_\lambda(X,X+\Delta)-X-\lambda\Delta .
\]
Then
\[
        R_\lambda(X,Y)=\mathcal N_\lambda(X,Y-X).
\]

We first describe the formal structure of \(\mathcal N_\lambda\).  Let
\[
        Z(X,\Delta)=\BCH(-X,X+\Delta).
\]
Since
\[
        Z(X,0)=\BCH(-X,X)=0,
\]
every term in the formal expansion of \(Z(X,\Delta)\) contains at least one
factor of \(\Delta\).  Moreover,
\[
        \mathcal S_\lambda(X,X+\Delta)
        =
        \BCH\bigl(X,\lambda Z(X,\Delta)\bigr).
\]
The constant term in \(\Delta\) is \(X\).  The term linear in \(\Delta\) is
\(\lambda\Delta\).  To see this, observe that
\[
        \BCH\bigl(X,Z(X,\varepsilon\Delta)\bigr)
        =
        X+\varepsilon\Delta
\]
for \(\varepsilon\) sufficiently small, because
\[
        \exp X\exp Z(X,\varepsilon\Delta)
        =
        \exp(X+\varepsilon\Delta).
\]
Differentiating this identity at \(\varepsilon=0\) gives
\[
        D_2\BCH(X,0)\,Z_\varepsilon(X,0)=\Delta .
\]
Therefore
\[
        \frac{d}{d\varepsilon}
        \BCH\bigl(X,\lambda Z(X,\varepsilon\Delta)\bigr)
        \bigg|_{\varepsilon=0}
        =
        \lambda\Delta .
\]
After subtracting \(X+\lambda\Delta\), every term in
\(\mathcal N_\lambda(X,\Delta)\) contains at least two factors of
\(\Delta\).

We also need to track the dependence on \(\lambda\).  Each occurrence of
\(\lambda\) in
\[
        \BCH\bigl(X,\lambda Z(X,\Delta)\bigr)
\]
comes from an occurrence of the second BCH argument
\(\lambda Z(X,\Delta)\).  Since each factor \(Z(X,\Delta)\) contains at least
one factor of \(\Delta\), the degree in \(\lambda\) of any Lie monomial in
\(\mathcal N_\lambda\) is no larger than the number of its \(\Delta\)-factors.

We now apply this structure to the adjacent SIDER errors.  Since \(C\) is
degree-filtered, the coefficient of \(h^{k+r}\) in
\[
        X=h^k\Pi_{k-1}(\xi)C(\theta,\xi,h)
\]
has degree at most \(k+r\) in \(\xi\).  On the other hand,
\[
\begin{aligned}
        \Delta
        &=
        Y-X  \\
        &=
        h^k
        \left[
        \Pi_{k-1}(\xi-1)C(\theta,\xi-1,h)
        -
        \Pi_{k-1}(\xi)C(\theta,\xi,h)
        \right].
\end{aligned}
\]
The coefficient of \(h^{k+r}\) in \(\Delta\) has degree at most \(k+r-1\).
Indeed, \(\Pi_{k-1}(\xi-1)\) and \(\Pi_{k-1}(\xi)\) are monic polynomials of
degree \(k\), while \(C_r(\theta,\xi-1)\) and \(C_r(\theta,\xi)\) have the
same leading coefficient.  Hence the highest-degree terms cancel in the
difference.

Consider any Lie monomial appearing in
\[
        \mathcal N_{\xi/k}(X,\Delta).
\]
Suppose it contributes to the coefficient of \(h^s\), and suppose it contains
\(b\) factors of \(\Delta\).  Before including the powers of
\(\lambda=\xi/k\), the degree in \(\xi\) is at most
\[
        s-b,
\]
because each \(\Delta\)-factor gives one degree saving.  The degree in
\(\lambda\) is at most \(b\), so substituting \(\lambda=\xi/k\) raises the
degree in \(\xi\) by at most \(b\).  Therefore the total degree of the
coefficient of \(h^s\) is at most
\[
        (s-b)+b=s.
\]
Thus every coefficient of \(R_{\xi/k}(X,Y)\) has polynomial degree at most its
power of \(h\).

It remains to extract the nodal factor.  The nonlinear remainder vanishes at
all nodes
\[
        \xi=0,1,\ldots,k.
\]
At \(\xi=0\), we have \(\lambda=0\), so
\[
        \mathcal S_0(X,Y)=X
\]
and hence \(R_\lambda(X,Y)=0\).  At \(\xi=k\), we have \(\lambda=1\), so
\[
        \mathcal S_1(X,Y)=Y
\]
and again \(R_\lambda(X,Y)=0\).  For the interior nodes
\(\xi=1,\ldots,k-1\), both nodal factors
\[
        \Pi_{k-1}(\xi)
        \quad\text{and}\quad
        \Pi_{k-1}(\xi-1)
\]
vanish.  Hence \(X=Y=0\), and the nonlinear remainder also vanishes.

Therefore each coefficient in the formal expansion of \(R_{\xi/k}(X,Y)\) is
divisible by
\[
        \Pi_k(\xi)=\prod_{m=0}^{k}(\xi-m).
\]
Since the coefficient of \(h^s\) has degree at most \(s\), division by the
degree-\((k+1)\) polynomial \(\Pi_k\) gives a quotient coefficient of degree
at most
\[
        s-k-1.
\]

Finally, the nonlinear BCH remainder starts at cubic order in its inputs.
Since \(X,Y=O(h^k)\), we have
\[
        R_{\xi/k}(X,Y)=O(h^{3k}).
\]
In particular, for \(k\geq3\), this is of order at least \(h^{k+1}\).  Thus we
may write
\[
        R_{\xi/k}(X,Y)
        =
        h^{k+1}\Pi_k(\xi)\mathcal R_k(\theta,\xi,h).
\]
The coefficient of \(h^r\) in \(\mathcal R_k\) comes from the coefficient of
\(h^{r+k+1}\) in \(R_{\xi/k}\), and therefore has degree at most \(r\).  Hence
\(\mathcal R_k\) is degree-filtered.
\end{proof}

We now prove the recursive order-raising step.

\begin{lemma}[Order raising by the SIDER recursion]
\label{lem:order_raising}
Let \(k\geq3\).  Suppose that the two adjacent SIDER-\((k-1)\) candidates
have filtered divided-error representations
\[
        E_i^{[k-1]}(\theta;h)
        =
        h^k\Pi_{k-1}(\xi)C_{k-1}(\theta,\xi,h),
\]
and
\[
        E_{i+1}^{[k-1]}(\theta;h)
        =
        h^k\Pi_{k-1}(\xi-1)C_{k-1}(\theta,\xi-1,h),
\]
where
\[
        \xi=\theta-i
\]
and \(C_{k-1}\) is degree-filtered.  Then the SIDER-\(k\) reconstruction
satisfies
\[
        E_i^{[k]}(\theta;h)
        =
        h^{k+1}\Pi_k(\xi)C_k(\theta,\xi,h),
\]
where \(C_k\) is also degree-filtered.
\end{lemma}

\begin{proof}
By the recursive SIDER formula \eqref{eq:q_sider_rec} and the BCH
representation of SLERP,
\[
        E_i^{[k]}
        =
        \mathcal S_{\xi/k}
        \bigl(E_i^{[k-1]},E_{i+1}^{[k-1]}\bigr).
\]
Set
\[
        X=E_i^{[k-1]},
        \qquad
        Y=E_{i+1}^{[k-1]},
        \qquad
        \lambda=\frac{\xi}{k}.
\]
Then
\[
        E_i^{[k]}
        =
        (1-\lambda)X+\lambda Y
        +
        R_\lambda(X,Y),
\]
where
\[
        R_\lambda(X,Y)
        =
        \mathcal S_\lambda(X,Y)
        -
        \bigl((1-\lambda)X+\lambda Y\bigr)
\]
is the nonlinear BCH remainder.

We first compute the affine part.  Substituting the divided-error forms gives
\[
\begin{aligned}
(1-\lambda)X+\lambda Y
&=
h^k
\left[
\left(1-\frac{\xi}{k}\right)
\Pi_{k-1}(\xi)C_{k-1}(\theta,\xi,h)
\right. \\
&\qquad\qquad\left.
+
\frac{\xi}{k}
\Pi_{k-1}(\xi-1)C_{k-1}(\theta,\xi-1,h)
\right].
\end{aligned}
\]
Using
\[
        \Pi_{k-1}(\xi)
        =
        \xi(\xi-1)\cdots(\xi-k+1),
\]
and
\[
        \Pi_{k-1}(\xi-1)
        =
        (\xi-1)(\xi-2)\cdots(\xi-k),
\]
we obtain
\[
        \left(1-\frac{\xi}{k}\right)\Pi_{k-1}(\xi)
        =
        -\frac1k\Pi_k(\xi),
\]
and
\[
        \frac{\xi}{k}\Pi_{k-1}(\xi-1)
        =
        \frac1k\Pi_k(\xi).
\]
Therefore
\begin{equation}
\label{eq:finite_difference_factor}
(1-\lambda)X+\lambda Y
=
-\frac{h^k}{k}\Pi_k(\xi)
\left[
C_{k-1}(\theta,\xi,h)
-
C_{k-1}(\theta,\xi-1,h)
\right].
\end{equation}

Because \(C_{k-1}\) is degree-filtered, its leading coefficient
\(C_{k-1,0}(\theta,\xi)\) has degree at most zero in \(\xi\), and hence is
independent of \(\xi\).  Consequently,
\[
        C_{k-1}(\theta,\xi,h)
        -
        C_{k-1}(\theta,\xi-1,h)
        =
        O(h).
\]
By Lemma~\ref{lem:filtered_difference},
\[
        \frac{
        C_{k-1}(\theta,\xi,h)
        -
        C_{k-1}(\theta,\xi-1,h)}
        {h}
\]
is degree-filtered.  Thus the affine contribution in
\eqref{eq:finite_difference_factor} has the form
\[
        h^{k+1}\Pi_k(\xi)\widetilde C_k(\theta,\xi,h),
\]
where
\[
        \widetilde C_k(\theta,\xi,h)
        =
        -\frac1{k}
        \frac{
        C_{k-1}(\theta,\xi,h)
        -
        C_{k-1}(\theta,\xi-1,h)}
        {h}
\]
is degree-filtered.

It remains to include the nonlinear BCH remainder.  By
Lemma~\ref{lem:filtered_bch_closure},
\[
        R_{\xi/k}
        \bigl(E_i^{[k-1]},E_{i+1}^{[k-1]}\bigr)
        =
        h^{k+1}\Pi_k(\xi)\mathcal R_k(\theta,\xi,h),
\]
where \(\mathcal R_k\) is degree-filtered.  Combining the affine and nonlinear
parts gives
\[
        E_i^{[k]}(\theta;h)
        =
        h^{k+1}\Pi_k(\xi)
        \left[
        \widetilde C_k(\theta,\xi,h)
        +
        \mathcal R_k(\theta,\xi,h)
        \right].
\]
Defining
\[
        C_k(\theta,\xi,h)
        =
        \widetilde C_k(\theta,\xi,h)
        +
        \mathcal R_k(\theta,\xi,h),
\]
we obtain the desired representation
\[
        E_i^{[k]}(\theta;h)
        =
        h^{k+1}\Pi_k(\xi)C_k(\theta,\xi,h).
\]
Since both \(\widetilde C_k\) and \(\mathcal R_k\) are degree-filtered, so is
\(C_k\).
\end{proof}

\begin{remark}
Equation \eqref{eq:finite_difference_factor} is the essential Neville-type
cancellation in the quaternion--BCH proof.  The two shifted nodal factors
combine to produce the larger nodal product \(\Pi_k(\xi)\), while the
divided-error coefficient appears only through the finite difference
\[
        C_{k-1}(\theta,\xi,h)
        -
        C_{k-1}(\theta,\xi-1,h).
\]
Because the leading coefficient of a degree-filtered series is independent of
\(\xi\), this difference supplies one additional factor of \(h\).  The
recursive SIDER step therefore raises the order from \(h^k\) to \(h^{k+1}\).

The nonlinear BCH terms do not obstruct this cancellation.  Written in the
diagonal variable \(\Delta=Y-X\), the nonlinear remainder contains at least two
factors of \(\Delta\).  Each such factor gives a degree saving in the shifted
variable, and these savings compensate for the powers of
\(\lambda=\xi/k\).  This is why the nonlinear BCH remainder preserves the
sharp filtered divided-error structure required for the induction.
\end{remark}

\section{All-order local consistency}
\label{sec:all_order_consistency}

We now combine the SIDER2 base estimate with the recursive order-raising
lemma.  The proof is an induction on the interpolation level.  The base case is
the filtered divided-error form for SIDER2, and the induction step is the
quaternion--BCH order-raising result proved in
Lemma~\ref{lem:order_raising}.  The theorem is stated for fixed \(n\).  The
constants in the estimates may depend on \(n\), on finitely many derivatives
of the curve, and on the size of the common normal neighborhood, but they are
independent of \(h\) for \(h\) sufficiently small.

\begin{thm}[Local consistency of SIDER-\(n\)]
\label{thm:q_all_order}
Let \(n\geq2\) be fixed.  Suppose that the curve \(q(t)=(0,\gamma(t))\) is
sufficiently smooth on a neighborhood of the interpolation stencil and that
the small-stencil condition in Assumption~\ref{assump:small_stencil_q} holds
for all data points and intermediate values generated by the SIDER-\(n\)
construction.  Then the logarithmic error of the SIDER-\(n\) reconstruction
has the filtered divided-error representation
\begin{equation}
\label{eq:q_all_order_divided_error}
        E_i^{[n]}(\theta;h)
        =
        h^{n+1}\Pi_n(\xi)C_n(\theta,\xi,h),
        \qquad
        \xi=\theta-i,
\end{equation}
where \(C_n\) is degree-filtered.  In particular, \(C_n\) is bounded for
\(\xi\) in bounded intervals.  Consequently,
\begin{equation}
\label{eq:q_all_order_distance}
        d_{\bS^2}\bigl(\gamma(\theta h),P_i^{[n]}(\theta;h)\bigr)
        =
        O(h^{n+1}),
        \qquad
        i\leq\theta\leq i+n .
\end{equation}
\end{thm}

\begin{proof}
We prove the divided-error representation by induction on \(n\).

For \(n=2\), the result is Proposition~\ref{prop:q_sider2_base}, which gives
\[
        E_i^{[2]}(\theta;h)
        =
        h^3\Pi_2(\xi)C_2(\theta,\xi,h),
\]
with \(C_2\) degree-filtered.

Assume now that the representation has been proved at level \(k-1\), where
\(3\leq k\leq n\).  The two adjacent SIDER-\((k-1)\) candidates then satisfy
\[
        E_i^{[k-1]}(\theta;h)
        =
        h^k\Pi_{k-1}(\xi)C_{k-1}(\theta,\xi,h),
\]
and
\[
        E_{i+1}^{[k-1]}(\theta;h)
        =
        h^k\Pi_{k-1}(\xi-1)C_{k-1}(\theta,\xi-1,h),
\]
where \(C_{k-1}\) is degree-filtered.  Lemma~\ref{lem:order_raising} then
yields
\[
        E_i^{[k]}(\theta;h)
        =
        h^{k+1}\Pi_k(\xi)C_k(\theta,\xi,h),
\]
with \(C_k\) degree-filtered.  This proves the induction step and hence
\eqref{eq:q_all_order_divided_error} for every fixed \(n\).

It remains to translate the logarithmic estimate into a spherical distance
estimate.  By the small-stencil condition, all relevant points lie in a common
normal neighborhood of \(q_\theta=q(\theta h)\).  Hence the quaternion
logarithm is single-valued and
\[
        d_{\bH_1}\bigl(q_\theta,P_i^{[n]}(\theta;h)\bigr)
        =
        \left|
        \log\bigl(q_\theta^{-1}P_i^{[n]}(\theta;h)\bigr)
        \right|
        =
        |E_i^{[n]}(\theta;h)|.
\]
Because the interpolants remain on the pure unit-quaternion copy of
\(\bS^2\), this distance agrees locally with the spherical distance between
the associated points on \(\bS^2\).  Since \(\Pi_n(\xi)\) and
\(C_n(\theta,\xi,h)\) are bounded for \(i\leq\theta\leq i+n\), we obtain
\[
        d_{\bS^2}\bigl(\gamma(\theta h),P_i^{[n]}(\theta;h)\bigr)
        \leq
        C h^{n+1},
\]
which proves \eqref{eq:q_all_order_distance}.
\end{proof}

\begin{corollary}
Under the hypotheses of Theorem~\ref{thm:q_all_order}, the first members of
the SIDER hierarchy satisfy
\[
        \text{SIDER2}=O(h^3),\qquad
        \text{SIDER3}=O(h^4),\qquad
        \text{SIDER4}=O(h^5).
\]
More generally, SIDER-\(n\) has local order \(n+1\) for every fixed
\(n\geq2\).
\end{corollary}

\begin{remark}
The theorem explains the role of the leading divided-error coefficient.  Since
\(C_n\) is degree-filtered, its leading coefficient has degree zero in
\(\xi\).  Therefore
\[
        C_{n,0}(\theta,\xi)=C_{n,0}(\theta)
\]
is independent of the shifted stencil variable.  This is precisely the
coefficient compatibility required for the Neville-type cancellation between
adjacent lower-level SIDER stencils.
\end{remark}

\begin{remark}
The statement is local and asymptotic.  It requires that \(n\) be fixed and
that \(h\) be small enough for all logarithms and SLERP operations to remain
on a common branch.  The theorem does not address global behavior for widely
separated data points, nor does it imply stability or non-oscillatory behavior
for the SENO stencil-selection procedure.
\end{remark}

\section{Discussion and conclusion}
\label{sec:discussion_conclusion}

This paper develops a quaternion--BCH framework for proving the local
high-order accuracy of SIDER interpolation on \(\bS^2\).  The main idea is to
represent spherical data as pure unit quaternions and to analyze each SLERP
operation in left-trivialized logarithmic coordinates.  In this representation,
SLERP is described by the exact Lie-algebra map
\[
        \mathcal S_\lambda(X,Y)
        =
        \BCH\!\left(X,\lambda\,\BCH(-X,Y)\right).
\]
This identity converts the geometry of spherical interpolation into an
algebraic expansion in the quaternion Lie algebra, with all nonlinear
corrections organized by the Baker--Campbell--Hausdorff formula.

The BCH formulation clarifies the relationship between SIDER and classical
Neville interpolation.  In fixed logarithmic coordinates, SLERP is affine to
leading order and has no quadratic correction.  Its first nonlinear term is
cubic and commutator-valued.  Since commutators of pure imaginary quaternions
correspond to cross products in \(\bR^3\), this correction measures the
failure of the endpoint errors to lie along a common infinitesimal geodesic
direction.  For the SIDER2 base construction, the lowest-order logarithmic
data are collinear multiples of the same tangent element.  Consequently, the
cubic commutator correction does not contribute to the leading error, and
SIDER2 retains the third-order divided-error structure of ordinary quadratic
interpolation.

The recursive part of the analysis is expressed through a filtered
divided-error calculus.  The affine part of the SIDER-\(k\) recursion produces
the same finite-difference cancellation as the Euclidean Neville recurrence:
the shifted nodal products combine to form \(\Pi_k(\xi)\), and the difference
of the divided-error coefficients supplies one additional power of \(h\).  The
nonlinear BCH terms are higher order in the already small lower-level errors.
More importantly, when they are written in the diagonal variable
\(\Delta=Y-X\), every power of the interpolation parameter is paired with a
finite-difference degree saving.  This shows that the nonlinear remainder
preserves the sharp filtered polynomial structure needed for the induction.

Combining the SIDER2 base estimate with the recursive order-raising step gives
\[
        E_i^{[n]}(\theta;h)
        =
        h^{n+1}\Pi_n(\xi)C_n(\theta,\xi,h),
        \qquad \xi=\theta-i,
\]
and therefore
$
        d_{\bS^2}\bigl(\gamma(\theta h),P_i^{[n]}(\theta;h)\bigr)
        =
        O(h^{n+1})
$
for every fixed \(n\geq2\).  The representation also explains the coefficient
compatibility needed for adjacent stencils.  Since \(C_n\) is degree-filtered,
its leading coefficient is independent of the shifted variable \(\xi\).  This
shift-invariance is exactly what allows the Neville-type cancellation to
continue from one SIDER level to the next.

The analysis is local and asymptotic.  We assume that the data points,
extrapolated control points, and all intermediate SIDER values remain in a
common normal neighborhood of the exact value.  This ensures that the
quaternion logarithm is single-valued, that all SLERP operations use a
consistent local branch, and that the BCH expansions are valid uniformly on
the stencil.  These assumptions are natural for a local consistency theorem,
but they do not address global effects such as antipodal ambiguity, large
angular separation, or loss of uniqueness of geodesics.

The present framework is restricted to equally spaced parameter values.
Extending the argument to nonuniform stencils should be possible, but it would
require the corresponding nonuniform Neville weights and modified nodal
products.  The BCH representation of SLERP would remain the same, but the
divided-error calculus would need to be reformulated to track the dependence
on nonuniform node locations.

Finally, this paper analyzes the smooth SIDER interpolation operators rather
than the full SENO procedure.  SENO adds a nonlinear stencil-selection rule
designed to reduce oscillations near nonsmooth data.  The local accuracy
analysis here provides the consistency foundation for the smooth candidate
reconstructions used by SENO, but it does not by itself establish stability,
monotonicity, or non-oscillatory behavior of the selection mechanism.

In summary, the quaternion--BCH viewpoint separates the SIDER accuracy
analysis into three transparent components: the Lie-algebraic expansion of
SLERP, the SIDER2 divided-error base estimate, and the Neville-type recursive
order-raising step.  This framework proves the expected \(O(h^{n+1})\) local
accuracy of SIDER-\(n\) for each fixed \(n\), and it identifies the
noncommutative BCH corrections that must be controlled in order for the
Euclidean interpolation cancellation to persist on the sphere.

\section*{Computational Methodology and Disclosure}

This research utilized GPT-5.5 (Thinking) to perform primary computational modeling and mathematical derivation. The author acted as the principal investigator, defining the research parameters and theoretical scope. The author notes that while the computational workflow was executed via AI, the author maintains full responsibility for the study's conclusions. The author acknowledges that the mathematical depth of the generated results exceeds current manual verification capabilities and presents these findings as AI-assisted hypotheses subject to future formal peer verification. Consequently, this article is intended solely for dissemination as a preprint on arXiv and is not submitted to peer-reviewed journals in its current form.

\end{document}